\documentclass[12 pt]{article}
\usepackage{stmaryrd}
\usepackage{times}
\usepackage{booktabs}
\usepackage{subfigure}

\usepackage[colorlinks,linkcolor=red]{hyperref}

\usepackage{pifont}
\usepackage{floatrow}
\usepackage{caption}
\usepackage{mathrsfs}
\usepackage[fleqn]{amsmath}
\usepackage{amsfonts,amsthm,amssymb,mathrsfs,bbding}
\usepackage{txfonts}
\usepackage{graphics,multicol}
\usepackage{graphicx}
\usepackage{color}
\usepackage{caption}
\usepackage{indentfirst}
\usepackage{cite}
\usepackage{latexsym,bm}
\usepackage{enumerate}
\usepackage{epsfig}
\evensidemargin=\oddsidemargin\topmargin=-1.5cm

\usepackage{amsmath}
\allowdisplaybreaks\allowdisplaybreaks[4]

\newtheorem{prob}{Problem}
\newtheorem{thm}{Theorem}
\newtheorem{lem}{Lemma}
\newtheorem{cor}{Corollary}

\theoremstyle{definition}

\newtheorem{remark}{Remark}

\newtheorem{exa}{Example}

\addtocounter{section}{0}
\begin{document}
\title{Infinite classes of strongly regular graphs derived from $GL(n,F_2)$\footnote{This work is supported
by NSFC Grant No. 11671344.}}
\author{{\small Lu Lu, \ \ Qiongxiang Huang\footnote{
Corresponding author.}, \ \ Jiangxia Hou\setcounter{footnote}{-1}\footnote{Email: lulu549588@hotmail.com (L. Lu), huangqx@xju.edu.cn (Q. Huang), jxhou@xju.edu.cn (J. Hou)}}\\[2mm]\scriptsize
College of Mathematics and Systems Science,
\scriptsize Xinjiang University, Urumqi, Xinjiang 830046, P.R.China}
\date{}
\maketitle {\flushleft\large\bf Abstract} It is known that the automorphism group of the elementary abelian $2$-group $Z_2^n$ is isomorphic to the general linear group $GL(n,F_2)$ of degree $n$ over $F_2$. Let $W$ be the collection of permutation matrices of order $n$. It is clear that $W\le GL(n,F_2)$. In virtue of this, we consider the Cayley graph $Cay(Z_2^n,S)$, where $S$ is the union of some orbits under the action of $W$. We call such graphs the orbit Cayley graphs over $Z_2^n$. In this paper, we give eight infinite families of strongly regular graphs among orbit Cayley graphs over $Z_2^n$, in which six families are new as we know. By the way, we formulate the spectra of orbit Cayley graphs as well.\begin{flushleft}
\end{flushleft}
\section{Introduction}
Let $\Gamma$ be a graph with vertex set $V=\{v_1,\ldots,v_n\}$. The \emph{adjacency matrix} of $\Gamma$, denoted by  $A(\Gamma)=(a_{i,j})_{n\times n}$, is the $n\times n$ matrix such that the $a_{i,j}=1$ if $v_i\sim v_j$ and $0$ otherwise. The eigenvalues of $A(G)$ are called the \emph{eigenvalues} of $\Gamma$ and the multiset of such eigenvalues together with their multiplicities is called the \emph{spectrum} of $\Gamma$, denoted by $\mathrm{Spec}(\Gamma)$. The \emph{distance} $d(v_i,v_j)$ (or $d_{ij}$ for short) between two vertices $v_i$ and $v_j$ is the length of a shortest path from $v_i$ to $v_j$. The largest distance in $\Gamma$ is the \emph{diameter} of $\Gamma$, denoted by $d(\Gamma)$, that is, $d(\Gamma)=\max\{d(v_i,v_j)\mid v_i,v_j\in V\}$. For $1\le l\le d(\Gamma)$, let $N^{(l)}(v_i)=\{v_j\mid d(v_i,v_j)=l\}$. Therefore, $V=N^{(1)}(v_i)\cup \cdots\cup N^{(d(\Gamma))}(v_i)$ is a partition of $V$ for any vertex $v_i$. Particularly, we alway write $N(v_i)$ for $N^{(1)}(v_i)$, which is just the neighborhood of $v_i$. Moreover, the size of $N(v_i)$ is the \emph{degree} of $v_i$, denoted by $d(v_i)$. A bijection $\sigma$: $V\rightarrow V$ is an \emph{automorphism} of $\Gamma$ if $v_i^{\sigma}\sim v_j^{\sigma}$ if and only if $v_i\sim v_j$. The set of all automorphisms forms the \emph{automorphism group} of $\Gamma$, denoted by $Aut(\Gamma)$. The graph $\Gamma$ is \emph{vertex transitive} if for any vertex $v_i$ there exists $\sigma\in Aut(\Gamma)$ such that $v_1^{\sigma}=v_i$.

 In 1970, Doob \cite{Doob} started the investigation of graphs with few distinct eigenvalues. Such graphs has been proved to possess nice combinatorial properties and rich structures \cite{Dam}. It is clear that the connected graphs with exactly two distinct eigenvalues are complete graphs. However, it is far from been solving to characterize the graphs with exactly three distinct eigenvalues. The non-regular graphs with three distinct eigenvalues are studied by many mathematicians, see, for example, \cite{Dam} and \cite{Koolen}. A connected regular graph with exactly three distinct eigenvalues must be a so called strongly regular graph \cite{Godsil}. A \emph{strongly regular graph} with parameters $(n,r,\lambda,\mu)$ is an $r$ regular graph on $n$ vertices in which any two adjacent vertices have exactly $\lambda$ common neighbors and any two non-adjacent vertices have exactly $\mu$ common neighbours. The characterization of strongly regular graphs is a classical problem in algebraic graph theory which has caused wide public attention and numerous results are obtained, especially, on the strongly regular Cayley graphs, we refer the reader to \cite{Calderbank, Feng, Ge, Momihara} and references therein. Let $G$ be a group and $S$ a subset of $G$. The \emph{Cayley graph} $X=Cay(G,S)$ is the graph such that $V(X)=G$ and two vertices $x$ and $y$ are adjacent if $x^{-1}y\in S$. If $S$ contains no identity element of $G$ and $S^{-1}=S$, then $X$ is an undirected simple graph. The \emph{circulant graphs} are the Cayley graphs over cyclic groups. The complete characterization of strongly regular circulant graphs are independently given by Bridges and Mena \cite{Bridges}, Ma \cite{Ma}, and partially by Maru\v{s}i\v{c} \cite{Marusic}. For other abelian groups, Leifman and Muzychuk \cite{Leifman} classified strongly regular Cayley graphs on $Z_{p^n}\times Z_{p^n}$ for a prime $p$. Recently, the complete characterization of \emph{minimal Cayley graphs} (that is Cayley graphs $Cay(G,S)$ such that $S$ is a minimal generating set of $G$) over abelian groups was given by \cite{Stefko}.

In this paper, we focus on the Cayley graphs over the elementary abelian $2$-group $Z_2^n$. It is clear that the automorphism group of $Z_2^n$ is isomorphic to the general linear group $GL(n,F_2)$. Let $W$ be the collection of permutation matrices of order $n$. It is clear that $W\le GL(n,F_2)$. Note that $Z_2^n$ can be viewed as the vector space of dimension $n$ over the field $F_2$. Therefore, each element $v$ in $Z_2^n$ is a vector $v=(a_1,\ldots,a_n)^T$ such that $a_i\in F_2$. Let $|v|$ be the number of $1$'s in $v$, that is, $|v|=\sum_{i=1}^na_i$. Denote by $O_i=\{v\in Z_2^n\mid |v|=i\}$ for $0\le i\le n$. It is clear that $Pv\in O_i$ for any $P\in W$ and $v\in O_i$. Moreover, for any two vectors $u,v\in O_i$, there exists $P\in W$ such that $Pu=v$. Therefore, each $O_i$ is an orbit of $Z_2^n$ under the action of $W$. A Cayley graph $Cay(Z_2^n,S)$ is called an \emph{orbit Cayley graph} if $S$ is the union of some $O_i$. In this paper, we present a simpler criterion for a Cayley graph to be strongly regular. By using this criterion and some combination identities, we construct eight infinite families of strongly regular graphs among the orbit Cayley graphs over $Z_2^n$, in which six families are new as we know. By the way, we formulate the spectra of orbit Cayley graphs over $Z_2^n$. At last, we propose two research problems.

\section{Orbit Cayley graphs over $Z_2^n$}
We start with the connectivities of orbit Cayley graphs. Since $Cay(Z_2^n,O_0)$ is the union of $2^n$ copies of loops and $Cay(Z_2^n,O_n)$ is the union of $2^{n-1}$ copies of $K_2$, we exclude them when we consider the connectivity of an orbit Cayley graph.
\begin{lem}\label{lem-z-1}
For $i\not\in\{0,n\}$, the set $O_i$ is a generating set of $Z_2^n$ if and only if $i$ is odd.
\end{lem}
\begin{proof}
If $i$ is even, then $O_i$ generates a subgroup in which each element $v$ satisfies that $|v|$ is even. The necessity follows. Now suppose that $i$ is an odd. If $i=1$ then $O_1$ is exactly a base of $Z_2^n$, and so is a generating set. If $i=3$, then $O_3$ contains $x_1=(1110\cdots 0)^T$, $x_2=(1101\cdots 0)^T$ and $x_3=(1011\cdots 0)^T$. We have $e_1=(1000\cdots 0)^T=x_1+x_2+x_3$. Since $O_3$ is an orbit under $W$, we have $O_1\subset\langle O_3\rangle$. Thus $O_3$ generates $Z_2^n$. Similarly we can prove that $O_i$ generates $Z_2^n$ for other odd number.
\end{proof}
It is well known that a Cayley graph $Cay(G,S)$ is connected if and only if $S$ is a generating set of $G$, that is, $\langle S\rangle=G$. The following result is immediate from Lemma \ref{lem-z-1}.
\begin{cor}\label{cor-z-1}
The orbit Cayley graph $Cay(Z_2^n,S)$ ($S\ne O_0$ or $O_n$) is connected if and only if $S$ contains $O_i$ for some odd $i$.
\end{cor}
In what follows, we focus on the spectra of connected orbit Cayley graphs. Note that the eigenvalues of Cayley graphs over an abelian group have been given by Babai.

\begin{lem}[\cite{Babai}]\label{lem-z-2}
Let $G$ be an abelian group of order $n$ and $S$ a subset of $G$ such that $1\not\in S$ and $S^{-1}=S$. If $\chi_1,\ldots,\chi_n$ are all irreducible characters of $G$, then the eigenvalues of the Cayley graph $Cay(G,S)$ are $\lambda_i=\sum_{s\in S}\chi_i(s)$ for $1\le i\le n$.
\end{lem}
To get the eigenvalues of orbit Cayley graphs over $Z_2^n$, we should know the irreducible characters of $Z_2^n$.
\begin{lem}[\cite{Jean}]\label{lem-z-3}
The irreducible characters of $Z_2^n$ are $\chi_{i_1,\ldots,i_n}$ for $i_j\in\{0,1\}$ and $1\le j\le n$, where $\displaystyle \chi_{i_1,\ldots,i_n}(v)=(-1)^{a_1i_1+\cdots+a_ni_n}$
 for $v=(a_1,\ldots,a_n)^T\in Z_2^n$.
\end{lem}
For a group $G$ and a subset $T$, denote by $\chi(T)=\sum_{t\in T}\chi(t)$ for a character $\chi$. By immediate calculation, we get the following result.
\begin{lem}\label{lem-z-4}
Let $\chi_{i_1,\ldots,i_n}$ be an irreducible character of $Z_2^n$ and $O_i$ an orbit. If $i_1+\cdots+i_n=k$, then $\displaystyle \chi_{i_1,\ldots,i_n}(O_i)=\sum_{j=0}^i{k\choose j}{n-k\choose i-j}(-1)^j$.
\end{lem}
\begin{proof}
By definition, we have $\chi_{i_1,\ldots,i_n}(O_i)=\sum_{v\in O_i}\chi_{i_1,\ldots,i_n}(v)$. Denote by $\Lambda_j^i=\{v=(a_1,\ldots,a_n)^T\in O_i\mid a_1i_1+\cdots+a_ni_n=j\}$. Therefore, for each $v\in \Lambda_j^i$, we have $\chi_{i_1,\ldots,i_n}(v)=(-1)^j$ and thus $\chi_{i_1,\ldots,i_n}(O_i)=\sum_{j=0}^{r(i,k)}|\Lambda_j^i|(-1)^j$, where $r(i,k)=\min\{i,k\}$. It is not hard to verify that $|\Lambda_j^i|={k\choose j}{n-k\choose i-j}$. Thus, we have
\[\displaystyle \chi_{i_1,\ldots,i_n}(O_i)=\sum_{j=0}^{r(i,k)}{k\choose j}{n-k\choose i-j}(-1)^j=\sum_{j=0}^i{k\choose j}{n-k\choose i-j}(-1)^j.\]
\end{proof}
\begin{thm}\label{thm-z-1}
Let $Cay(Z_2^n,S)$ be an orbit Cayley graph over $Z_2^n$. If $S=\cup_{i\in I}O_i$ for some subset $I\subseteq\{1,\ldots,n\}$, then the eigenvalues of $Cay(Z_2^n,S)$ consist of $\lambda_k(S)$ with multiplicity ${n\choose k}$ for $0\le k\le n$, where
\[\displaystyle \lambda_k(S)=\sum_{i\in I}\sum_{j=0}^{i}{k\choose j}{n-k\choose i-j}(-1)^j.\]
\end{thm}
\begin{proof}
By Lemma \ref{lem-z-2}, the eigenvalues of $Cay(Z_2^n,S)$ are given by
\[\chi_{i_1,\ldots,i_n}(S)=\sum_{i\in I}\chi_{i_1,\ldots,i_n}(O_i)\]
where $\chi_{i_1,\ldots,i_n}$ are irreducible characters of $Z_2^n$. Denote by $\Upsilon_k=\{\chi_{i_1,\ldots,i_n}\mid i_1+\cdots+i_n=k\}$ for $0\le k\le n$. By Lemma \ref{lem-z-4}, we have
\[\chi_{i_1,\ldots,i_n}(S)=\sum_{i\in I}\sum_{j=0}^i{k\choose j}{n-k\choose i-j}(-1)^j\]
for any $\chi_{i_1,\ldots,i_n}\in \Upsilon_k$. It means that all characters in $\Upsilon_k$ correspond to the same eigenvalue, denoted by $\lambda_k$. Note that $|\Upsilon_k|={n\choose k}$. It means that the multiplicity of $\lambda_k$ is ${n\choose k}$. It completes the proof.
\end{proof}
At a glance of Theorem \ref{thm-z-1}, one may think that any orbit Cayley graph over $Z_2^n$ has exactly $n+1$ distinct eigenvalues. However, the values $\lambda_k$ and $\lambda_{k'}$ may be the same for distinct integers $k$ and $k'$. Thus,  Theorem \ref{thm-z-1} just implies the following result.
\begin{cor}\label{cor-z-2}
If $Cay(Z_2^n,S)$ is an orbit Cayley graph, then it has at most $n+1$ distinct eigenvalues.
\end{cor}
Note that, if $S=Z_2^n\setminus \{\mathbf{0}\}$, then the orbit Cayley graph $Cay(Z_2^n,S)$ is isomorphic to the complete graph, which has exactly two distinct eigenvalues. In general, we pose the following problem.
\begin{prob}\label{conj-d-1}
If there exists an orbit Cayley graph over $Z_2^n$ with exactly $m$ distinct eigenvalues for any integer $m$ such that $2\le m\le n+1$ ?
\end{prob}
The answer is positive for $m=2$. The following examples imply that the answer is also positive for $m=n+1$, $\lfloor n/2\rfloor+2$ and $3$, respectively. Note that, if an $r$-regular graph has eigenvalues $r>\lambda_2\ge\cdots\ge\lambda_n$, then its complement has eigenvalues $n-1-r>-1-\lambda_n>\cdots>-1-\lambda_2$ \cite{Cvetkovic}.
\begin{exa}\label{exa-1}
We consider the orbit Cayley graph $X=Cay(Z_2^n,S)$ where $S=Z_2^n\setminus(\{\mathbf{0}\}\cup O_1)$. It is clear that $\overline{X}=Cay(Z_2^n,O_1)$. By Theorem \ref{thm-z-1}, the spectrum of $\overline{X}$ consists of
\[\lambda_k(O_1)=\sum_{j=0}^1{k\choose j}{n-k\choose i-j}(-1)^j={k\choose 0}{n-k\choose 1}-{k\choose 1}{n-k\choose 0}=n-2k\]
with multiplicity ${n\choose k}$ for $0\le k\le n$. Note that $\overline{X}$ is $n$-regular. The spectrum of $X$ consists of $\mu_0=2^n-1-n$ with multiplicity $1$ and $\mu_k=2k-1-n$ with multiplicity ${n\choose k}$ for $1\le k\le n$. Therefore, $X$ and $\overline{X}$ are both connected orbit Cayley graphs having $n+1$ distinct eigenvalues.
\end{exa}
\begin{exa}\label{exa-2}
We consider the orbit Cayley graph $X=Cay(Z_2^n,O_2)$. By Corollary \ref{cor-z-1}, $X$ is disconnected and thus $\overline{X}$ is connected. By Theorem \ref{thm-z-1}, the eigenvalues of $X$ consist of
\[\lambda_k(O_2)=\sum_{j=0}^2{k\choose j}{n-k\choose i-j}(-1)^j=\frac{n(n-1)-4k(n-k)}{2}\]
with multiplicity ${n\choose k}$ for $0\le k\le n$. Therefore, the eigenvalues of $\overline{X}$ consist of $\mu_0=2^n-1-\frac{n(n-1)}{2}$ with multiplicity $1$ and $\mu_k=-1-\frac{n(n-1)-4k(n-k)}{2}$ with multiplicity ${n\choose k}$ for $1\le k\le n$. Note that $\lambda_k=-1-\frac{n(n-1)-4k(n-k)}{2}=-1-\frac{n(n-1)-4k'(n-k')}{2}=\lambda_{k'}$ if and only if $k'=k$ or $k'=n-k$. It means that the connected orbit graph $\overline{X}$ has exactly $\lfloor n/2\rfloor+2$ eigenvalues.
\end{exa}
\begin{exa}\label{exa-3}
We consider the orbit Cayley graph $X=Cay(Z_2^n,S)$ where $S=\cup_{i=1}^{n-1}O_i$. It is clear that $\overline{X}=Cay(Z_2^n,O_n)$ which is the union of $2^{n-1}$ copies of $K_2$. Therefore, the spectrum of $\overline{X}$ consists of $\pm 1$ with multiplicity $2^{n-1}$ and thus the spectrum of $X$ consists of $2^n-2$ with multiplicity $1$, $0$ with multiplicity $2^{n-1}$ and $-2$ with multiplicity $2^{n-1}-1$. It follows that $X$ has exactly three distinct eigenvalues.
\end{exa}
Note that the Cayley graphs having three distinct eigenvalues are strongly regular graphs. However, it is not easy to characterize the strongly regular orbit Cayley graphs by immediately calculating their spectra by applying Theorem \ref{thm-z-1}. In the rest of this paper, we will characterize some strongly regular orbit Cayley graphs by analyzing their structures. We need some combination identities.
\section{Combination identities}
In this part, we shall give some combination identities, which will be used to prove our main results. Throughout this paper, we always set ${s\choose t}=0$ if $t<0$ or $t>s$.
\begin{lem}\label{lem-z-5}
For a positive integer $n$, we have
\[\sum_{i=0}^n{n\choose 2i}=\sum_{i=0}^n{n\choose 2i+1}=2^{n-1}.\]
\end{lem}
\begin{proof}
It is clear that
\[\left\{\begin{array}{l}
\displaystyle (1+1)^n=\sum_{i=0}^n{n\choose i}=\sum_{i=0}^n{n\choose 2i}+\sum_{i=0}^n{n\choose 2i+1}\\
\displaystyle (1-1)^n=\sum_{i=0}^n{n\choose i}(-1)^i=\sum_{i=0}^n{n\choose 2i}-\sum_{i=0}^n{n\choose 2i+1}
\end{array}\right.\]
By summation the two equations, we get $2\sum_{i=0}^n{n\choose 2i}=2^n$ and thus $\sum_{i=0}^n{n\choose 2i}=2^{n-1}$. Subtracting the second equation from the first one, we have $2\sum_{i=0}^n{n\choose 2i+1}=2^n$ and thus $\sum_{i=0}^n{n\choose 2i+1}=2^{n-1}$.
\end{proof}
\begin{lem}\label{lem-z-6}
For a positive integer $m$, we have
\[\left\{\begin{array}{l}
\displaystyle \sum_{j=0}^{m}{4m\choose 4j}=2^{4m-2}+(-1)^m2^{2m-1}\\
\displaystyle \sum_{j=0}^{m}{4m\choose 4j+2}=2^{4m-2}-(-1)^m2^{2m-1}\\
\displaystyle \sum_{j=0}^{m}{4m\choose 4j+1}=\sum_{j=0}^{m}{4m\choose 4j+3}=2^{4m-2}\\
\end{array}\right.\]
\end{lem}
\begin{proof}
Let $i$ be the imaginary unit. Note that $1+i=\sqrt{2}e^{\pi i/4}$. We have
\begin{eqnarray*}
&&\displaystyle 2^{2m}\cos{m\pi}+(2^{2m}\sin{m\pi}) i=2^{2m}e^{m\pi i}=(\sqrt{2}e^{\pi i/4})^{4m}=(1+i)^{4m}\\
&=&\displaystyle \sum_{j=0}^m{4m\choose 4j}+\sum_{j=0}^m{4m\choose 4j+2}i^2+\sum_{j=0}^m{4m\choose 4j+1}i+\sum_{j=0}^m{4m\choose 4j+3}i^3\\
&=&\displaystyle \left(\sum_{j=0}^m{4m\choose 4j}-\sum_{j=0}^m{4m\choose 4j+2}\right)+\left(\sum_{j=0}^m{4m\choose 4j+1}-\sum_{j=0}^m{4m\choose 4j+3}\right)i.
\end{eqnarray*}
It leads to that
\begin{equation}\label{eq-x-1}\left\{\begin{array}{l}
\displaystyle \sum_{j=0}^m{4m\choose 4j}-\sum_{j=0}^m{4m\choose 4j+2}=2^{2m}\cos{m\pi}=(-1)^m2^{2m}\\
\displaystyle \sum_{j=0}^m{4m\choose 4j+1}-\sum_{j=0}^m{4m\choose 4j+3}=2^{2m}\sin{m\pi}=0.
\end{array}\right.\end{equation}
From Lemma \ref{lem-z-5}, we have
\begin{equation}\label{eq-x-2}
\sum_{j=0}^m{4m\choose 4j}+\sum_{j=0}^m{4m\choose 4j+2}=\sum_{j=0}^m{4m\choose 4j+1}+\sum_{j=0}^m{4m\choose 4j+3}=2^{4m-1}.
\end{equation}
Thus, the result follows from (\ref{eq-x-1}) and (\ref{eq-x-2}).
\end{proof}
Similarly, by respectively considering $(1+i)^{4m+1}$, $(1+i)^{4m+2}$ and $(1+i)^{4m+3}$, we get the following result.
\begin{lem}\label{lem-z-7}
For a positive integer $m$, we have\\
(a) $\displaystyle \sum_{j=0}^m{4m+1\choose 4j}=\sum_{j=0}^m{4m+1\choose 4j+1}=2^{4m-1}+(-1)^m2^{2m-1}$;\\
(b) $\displaystyle \sum_{j=0}^m{4m+1\choose 4j+2}=\sum_{j=0}^m{4m+1\choose 4j+3}=2^{4m-1}-(-1)^m2^{2m-1}$;\\
(c) $\displaystyle \sum_{j=0}^m{4m+2\choose 4j}=\sum_{j=0}^m{4m+2\choose 4j+2}=2^{4m}$;\\
(d) $\displaystyle \sum_{j=0}^m{4m+2\choose 4j+1}=2^{4m}+(-1)^m2^{2m}$;\\
(e) $\displaystyle \sum_{j=0}^m{4m+2\choose 4j+3}=2^{4m}-(-1)^m2^{2m}$;\\
(f) $\displaystyle \sum_{j=0}^m{4m+3\choose 4j}=\sum_{j=0}^m{4m+3\choose 4j+3}=2^{4m+1}-(-1)^m2^{2m}$;\\
(g) $\displaystyle \sum_{j=0}^m{4m+3\choose 4j+1}=\sum_{j=0}^m{4m+3\choose 4j+2}=2^{4m+1}+(-1)^m2^{2m}$.
\end{lem}
\begin{proof}
As similar as the proof of Lemma \ref{lem-z-6}, we have
\begin{eqnarray*}
&&\displaystyle 2^{2m}\cos{m\pi}+(2^{2m}\cos{m\pi}) i=(\sqrt{2}e^{\pi i/4})^{4m+1}=(1+i)^{4m+1}\\
&=&\displaystyle \sum_{j=0}^m{4m+1\choose 4j}+\sum_{j=0}^m{4m+1\choose 4j+2}i^2+\sum_{j=0}^m{4m+1\choose 4j+1}i+\sum_{j=0}^m{4m+1\choose 4j+3}i^3\\
&=&\displaystyle \left(\sum_{j=0}^m{4m+1\choose 4j}-\sum_{j=0}^m{4m+1\choose 4j+2}\right)+\left(\sum_{j=0}^m{4m+1\choose 4j+1}-\sum_{j=0}^m{4m+1\choose 4j+3}\right)i.
\end{eqnarray*}
It leads to that
\[\left\{\begin{array}{l}
\displaystyle \sum_{j=0}^m{4m+1\choose 4j}-\sum_{j=0}^m{4m+1\choose 4j+2}=2^{2m}\cos{m\pi}=(-1)^m2^{2m}\\
\displaystyle\sum_{j=0}^m{4m+1\choose 4j+1}-\sum_{j=0}^m{4m+1\choose 4j+3}=2^{2m}\cos{m\pi}=(-1)^m2^{2m}.
\end{array}\right.\]
From Lemma \ref{lem-z-5}, we have
\[
\sum_{j=0}^m{4m+1\choose 4j}+\sum_{j=0}^m{4m+1\choose 4j+2}=\sum_{j=0}^m{4m+1\choose 4j+1}+\sum_{j=0}^m{4m+1\choose 4j+3}=2^{4m}.
\]
Therefore, we have $\sum_{j=0}^m{4m+1\choose 4j}=\sum_{j=0}^m{4m+1\choose 4j}=2^{4m-1}+(-1)^m2^{2m-1}$ and $\sum_{j=0}^m{4m+1\choose 4j+2}=\sum_{j=0}^m{4m+1\choose 4j+3}=2^{4m-1}-(-1)^m2^{2m-1}$. It follows (a) and (b).

Similarly, we get (c) (d) and (e) when we consider $(1+i)^{4m+2}$ and get (f) and (g) when we consider $(1+i)^{4m+3}$.
\end{proof}
By applying Lemmas \ref{lem-z-6} and \ref{lem-z-7}, we get the following combination identity.
\begin{thm}\label{thm-z-2}
For two positive integers $k$ and $m$ such that $k\le m$, we have
\begin{equation}\label{eq-2-1}
\sum_{t=0}^m\sum_{j=0}^{2k}{4k\choose 2j}{4m-4k+1\choose 4t-2j+1}=2^{4m-2}+(-1)^m2^{2m-1}.
\end{equation}
\end{thm}
\begin{proof}
Note that $\{j\mid j\in Z, 0\le j\le 2k\}=\{2i\mid i\in Z, 0\le i\le k\}\cup \{2i+1\mid i\in Z, 0\le i\le k-1\}$. We have
\[\begin{array}{lll}
&&\displaystyle\sum_{t=0}^m\sum_{j=0}^{2k}{4k\choose 2j}{4m-4k+1\choose 4t-2j+1}\\
&=&\displaystyle\sum_{i=0}^{k}{4k\choose 4i}\sum_{t=0}^m{4(m-k)+1\choose 4(t-i)+1}+\sum_{i=0}^{k}{4k\choose 4i+2}\sum_{t=0}^m{4(m-k)+1\choose 4(t-i)-1}\\
&=&\displaystyle\sum_{i=0}^{k}{4k\choose 4i}\sum_{t'=-i}^{m-i}{4(m-k)+1\choose 4t'+1}+\sum_{i=0}^{k}{4k\choose 4i+2}\sum_{t'=-i}^{m-i}{4(m-k)+1\choose 4t'-1}\\
&=&\displaystyle\sum_{i=0}^{k}{4k\choose 4i}\sum_{t'=0}^{m-k}{4(m-k)+1\choose 4t'+1}+\sum_{i=0}^{k}{4k\choose 4i+2}\sum_{t'=0}^{m-k}{4(m-k)+1\choose 4t'-1},\\
\end{array}\]
the third equality holds because ${4(m-k)+1\choose 4t'+1}=0$ and ${4(m-k)+1\choose 4t'-1}=0$ if $t'<0$ or $t'>m-k$. From Lemmas \ref{lem-z-6} and \ref{lem-z-7}, we have
\[\begin{array}{lll}
&&\displaystyle\sum_{i=0}^{k}{4k\choose 4i}\sum_{t'=0}^{m-k}{4(m-k)+1\choose 4t'+1}+\sum_{i=0}^{k}{4k\choose 4i+2}\sum_{t'=0}^{m-k}{4(m-k)+1\choose 4t'-1}\\
&=&\displaystyle(2^{4k-2}+(-1)^k2^{2k-1})(2^{4(m-k)-1}+(-1)^{m-k}2^{2(m-k)-1})\\&&+(2^{4k-2}-(-1)^k2^{2k-1})(2^{4(m-k)-1}-(-1)^{m-k}2^{2(m-k)-1})\\
&=&\displaystyle2^{4m-2}+(-1)^m2^{2m-1}.
\end{array}\]

It completes the proof.
\end{proof}
By the similar methods, we obtain the following combination identities.
\begin{thm}\label{thm-z-3}
For two positive integers $k$ and $m$ such that $k\le m$, we have\\
(i) $\displaystyle 2\sum_{t=0}^m\sum_{j=0}^{2k}{4k+1\choose 2j}{4m-4k-1\choose 4t-2j}=2^{4m-2}+(-1)^m2^{2m-1}$;\\
(ii) $\displaystyle \sum_{t=0}^m\sum_{j=0}^{2k}{4k+2\choose 2j+1}{4m-4k-1\choose 4t-2j}=2^{4m-2}+(-1)^m2^{2m-1}$;\\
(iii) $\displaystyle 2\sum_{t=0}^m\sum_{j=0}^{2k+1}{4k+3\choose 2j+1}{4m-4k-3\choose 4t-2j-1}=2^{4m-2}+(-1)^m2^{2m-1}$;\\
(iv) $\displaystyle \sum_{t=0}^m\sum_{j=0}^{2k}{4k\choose 2j}{4m-4k+3\choose 4t-2j+1}=2^{4m}+(-1)^m2^{2m}$;\\
(v) $\displaystyle 2\sum_{t=0}^m\sum_{j=0}^{2k}{4k+1\choose 2j}{4m-4k+1\choose 4t-2j}=2^{4m}+(-1)^m2^{2m}$;\\
(vi) $\displaystyle \sum_{t=0}^m\sum_{j=0}^{2k}{4k+2\choose 2j+1}{4m-4k+1\choose 4t-2j}=2^{4m}+(-1)^m2^{2m}$;\\
(vii) $\displaystyle 2\sum_{t=0}^m\sum_{j=0}^{2k+1}{4k+3\choose 2j+1}{4m-4k-1\choose 4t-2j-1}=2^{4m}+(-1)^m2^{2m}$;\\
(viii) $\displaystyle \sum_{t=0}^m\sum_{j=0}^{2k}{4k\choose 2j}{4m-4k+3\choose 4t-2j}=2^{4m}-(-1)^m2^{2m}$;\\
(ix) $\displaystyle 2\sum_{t=0}^m\sum_{j=0}^{2k+1}{4k+3\choose 2j}{4m-4k-1\choose 4t-2j}=2^{4m}-(-1)^m2^{2m}$;\\
(x) $\displaystyle 2\sum_{t=0}^m\sum_{j=0}^{2k}{4k+1\choose 2j+1}{4m-4k+1\choose 4t-2j-1}=2^{4m}-(-1)^m2^{2m}$;\\
(xi) $\displaystyle \sum_{t=0}^m\sum_{j=0}^{2k}{4k+2\choose 2j+1}{4m-4k+1\choose 4t-2j+3}=2^{4m}-(-1)^m2^{2m}$.
\end{thm}
\begin{proof}
As similar as the proof of Theorem \ref{thm-z-2}, we prove these identities by two steps. At first, by using $\{j\mid j\in Z, 0\le j\le 2k\}=\{2i\mid i\in Z, 0\le i\le k\}\cup \{2i+1\mid i\in Z, 0\le i\le k-1\}$, the left side of each identity can be written as the sum of two parts. Next, by using the identities in Lemmas \ref{lem-z-6} and \ref{lem-z-7}, we obtain the identity.

For example, the left side of (i) can be written as
\[\begin{array}{lll}
&&\displaystyle \sum_{t=0}^m\sum_{j=0}^{2k}{4k+1\choose 2j}{4m-4k-1\choose 4t-2j}\\
&=&\displaystyle \sum_{i=0}^{4k}{4k+1\choose 4i}\sum_{t=0}^m{4(m-k)-1\choose 4(t-i)}+\sum_{i=0}^{4k}{4k+1\choose 4i+2}\sum_{t=0}^m{4(m-k)-1\choose 4(t-i)-2}\\
&=&\displaystyle \sum_{i=0}^{4k}{4k+1\choose 4i}\sum_{t'=0}^{m-k-1}{4(m-k-1)+3\choose 4t'}+\sum_{i=0}^{4k}{4k+1\choose 4i+2}\sum_{t'=0}^{m-k-1}{4(m-k-1)+3\choose 4t'+2}.
\end{array}\]
By Lemma \ref{lem-z-7} (a) (b) (f) and (g), we have
\[\begin{array}{lll}
&&\displaystyle \sum_{i=0}^{4k}{4k+1\choose 4i}\sum_{t'=0}^{m-k-1}{4(m-k-1)+3\choose 4t'}+\sum_{i=0}^{4k}{4k+1\choose 4i+2}\sum_{t'=0}^{m-k-1}{4(m-k-1)+3\choose 4t'+2}\\
&=&(2^{4k-1}+(-1)^k2^{2k-1})(2^{4(m-k-1)+1}-(-1)^{m-k-1}2^{2(m-k-1)})\\
&&+(2^{4k-1}-(-1)^k2^{2k-1})(2^{4(m-k-1)+1}+(-1)^{m-k-1}2^{2(m-k-1)})\\
&=&2^{4m-3}+(-1)^m2^{2m-2}.\end{array}\]
It leads to (i).

By these same processes, one can verify the other identities one by one.
\end{proof}
At last, we mention that  (i), (ii) and (iii) of Theorem \ref{thm-z-3} and the identity of Theorem \ref{thm-z-2} give four distinct combination representations but achieve the same value $2^{4m-2}+(-1)^m2^{2m-1}$;  the next four identities and the last four identities of Theorem \ref{thm-z-3} also give four distinct combination representations but achieve the same values $2^{4m}+(-1)^m2^{2m}$ and $2^{4m}-(-1)^m2^{2m}$, respectively. Such a properties are important to verify strongly regular orbit Cayley graphs in the next section.

\section{Strongly regular orbit Cayley graphs}
A connected $r$-regular graph on $n$ vertices is said to be \emph{strongly regular} with parameter $(n,r,\lambda,\mu)$ if two vertices $x$ and $y$ have $\lambda$ common neighbors whenever they are adjacent and $\mu$ common neighbors whenever they are not adjacent. A strongly regular graph is said to be \emph{trivial} if its complement is disconnected. It is clear that the complement of a trivial strongly regular graph is the union of some isomorphic complete graphs \cite{Godsil}. We first characterize the trivial strongly regular orbit Cayley graphs over $Z_2^n$.
\begin{thm}\label{thm-z-4}
Let $S^-=\cup_{i=1}^{n-1}O_i$ and $S^o=\cup_{i=1}^{\lceil n/2\rceil}O_{2i-1}$ be two subsets of $Z_2^n$. The orbit Cayley graph $Cay(Z_2^n,S)$ is a trivial strongly regular Cayley graph if and only if $S\in \{S^-,S^o\}$. Furthermore, $Cay(Z_2^n,S^-)$ is $(2^n, 2^n-2,2^n-4,2^n-2)$-strongly regular and $Cay(Z_2^n,S^o)$ is $(2^n,2^{n-1},0,2^{n-1})$-strongly regular.
\end{thm}
\begin{proof}
We have proved that $Cay(Z_2^n,S^-)$ is a trivial strongly regular graph in Example \ref{exa-3}. Similarly, one can verify that the complement of $Cay(Z_2^n,S^o)$ is $2K_{2^{n-1}}$ and thus it is a trivial strongly regular graph. Therefore, the sufficiency follows and we shall show the necessity in what follows. Let $X=Cay(Z_2^n,S)$ be a trivial strongly regular orbit Cayley graph. Therefore, $\overline{X}$ is disconnected and thus $\overline{S}=Z_2^{n}\setminus(\{\mathbf{0}\}\cup S)$ does not contain $O_i$ for any odd $i$ or $\overline{S}=O_n$. If the latter occurs, then we have $S=S^-$. If the former occurs, then $\overline{S}$ is the union of some $O_{2i}$. In fact, $\overline{S}$ contains all $O_{i}$ for even $i$ since otherwise the component $Cay(\langle\overline{S}\rangle,\overline{S})$ of $\overline{X}$ cannot be complete. It follows that $S=S^o$. Besides, one can easily get the parameters of $Cay(Z_2^n,S^-)$ and $Cay(Z_2^n,S^-)$ because $\overline{Cay(Z_2^n,S^-)}=2^{n-1}K_2$ and $\overline{Cay(Z_2^n,S^-)}=2K_{2^{n-1}}$.
\end{proof}

In what follows, we consider the non-trivial strongly regular orbit Cayley graphs. At first we give a criterion for a vertex transitive graph to be strongly regular.
\begin{lem}\label{lem-z-8}
Let $\Gamma$ be a connected vertex transitive graph on $n$ vertices. Then $\Gamma$ is strongly regular if and only if the partition $\Pi:$ $V(\Gamma)=\{v\}\cup N(v)\cup N^2(v)$ is an equitable partition for some $v\in V(\Gamma)$.
\end{lem}
\begin{proof}
It is known that if $\Gamma$ is strongly regular, then the partition $\Pi$: $V(\Gamma)=\{v\}\cup N(v)\cup N^2(v)$ is equitable for any $v\in V(\Gamma)$. The necessity follows. In what follows, we show the sufficiency. It is clear that $\Gamma$ is $r$-regular for some $r$ because it is vertex transitive. Since $\Pi$ is an equitable partition, assume the quotient matrix is
\[B_{\Pi}=\left(\begin{array}{ccc}0&b_{12}&b_{13}\\b_{21}&b_{22}&b_{23}\\b_{31}&b_{32}&b_{33}\end{array}\right),\]
where $b_{12}+b_{13}=b_{21}+b_{22}+b_{23}=b_{31}+b_{32}+b_{33}=r$.
For any two vertex $x,y$ such that $x\sim y$, there exists $\sigma\in Aut(\Gamma)$ such that $x^{\sigma}=v$ and thus $y^{\sigma}=v_1$ for some $v_1\in N(v)$. Therefore, we have $|N(x)\cap N(y)|=|(N(x)\cap N(y))^{\sigma}|=|N(x^{\sigma})\cap N(y^{\sigma})|=|N(v)\cap N(v_1)|=b_{22}$. Similarly, for any two vertices $x,y\in V(\Gamma)$ such that $x\not\sim y$, we have $|N(x)\cap N(y)|=|N(v)\cap N(v_2)|=b_{32}$, where $y^{\tau}=v_2\in N^2(v)$ for some $\tau\in Aut(\Gamma)$ and $x^{\tau}=v$. It means that $\Gamma$ is a $(n,r,b_{22},b_{32})$-strongly regular graph.
\end{proof}
For a subset $S$ of a group $G$, let $\overline{S}=G\setminus(S\cup\{1\})$. Denote by $C(v,S)=\{(x,y)\mid x,y\in S \textrm{ and } v=xy\}$. By applying Lemma \ref{lem-z-8}, we get a criterion for a Cayley graph to be strongly regular.

\begin{thm}\label{cor-z-3}
Let $S$ be a subset of a group $G$ such that the identity element $1\not\in S$ and $S=S^{-1}$. Then the Cayley graph $Cay(G,S)$ is $(|G|,|S|,\lambda,\mu)$-strongly regular if and only if $|C(v,S)|=\lambda$ for any $v\in S$ and $|C(v',S)|=\mu$ for any $v'\in \overline{S}$.
\end{thm}
\begin{proof}
At first, we will verify that $|N(v)\cap S|=|C(v,S)|$ for any $v\in G$. On the one hand, for any $x\in N(v)\cap S$, we have $x^{-1}v=y\in S$, that is, $v=xy$. It means that $(x,y)\in C(v,S)$ and thus $|N(v)\cap S|\le |C(v,S)|$. On the other hand, for any $(x,y)\in C(v,S)$, we have $x,y\in S$ and $v=xy$, that is, $x^{-1}v=y\in S$. It means that $x\in N(v)\cap S$ and thus $|C(v,S)|\le |N(v)\cap S|$.

Assume that $X=Cay(G,S)$ is $(|G|,|S|,\lambda,\mu)$-strongly regular. For any $v\in S$, we have $v\sim 1$ and thus $|N(v)\cap N(1)|=\lambda$. Therefore, by noticing $N(1)=S$, we have $|C(v,S)|=|N(v)\cap S|=|N(v)\cap N(1)|=\lambda$. 
For any $v'\in \overline{S}$, we have $v'\not\sim 1$ and thus $|N(v')\cap N(1)|=\mu$. Therefore, by noticing $N(1)=S$, we have
$|C(v',S)|=|N(v')\cap S|=|N(v')\cap N(1)|=\mu$.

Conversely, assume that $|C(v,S)|=\lambda$ for any $v\in S$ and $|C(v',S)|=\mu$ for any $v'\in \overline{S}$. It is clear that $X=Cay(G,S)$ is vertex transitive and $\Pi$: $V(X)=\{1\}\cup S\cup \overline{S}$ is a partition.
Moreover, we have  $|N(v)\cap S|=|C(v,S)|=\lambda$ for any $v\in S$ and $|N(v')\cap S|=|C(v',S)|=\mu$ for any $v'\in \overline{S}$. Therefore, we see that $\Pi$ is an equitable partition with quotient matrix
\[B_{\Pi}=\left(\begin{array}{ccc}0&|S|&0\\1&\lambda&|S|-1-\lambda\\0&\mu&|S|-\mu\end{array}\right)\begin{array}{c}\{1\}\\S\\\overline{S}\end{array}.\]
Thus, Lemma \ref{lem-z-8} implies that $X$ is $(|G|,|S|,\lambda,\mu)$-strongly regular.
\end{proof}

\begin{remark}
Easily to show that a vertex-transitive graph with diameter two is strongly regular if it is arc-transitive. Lemma \ref{lem-z-8} gives a sufficient and necessary condition for a vertex transitive graph to be strongly regular, and Theorem \ref{cor-z-3} also  gives a sufficient and necessary condition for a Cayley graph to be strongly regular.  Such a sufficient and necessary condition is essentially  to verify wether $V(X)=\{1\}\cup N(1)\cup N^2(1)$ is an equitable partition, which  is really weaker than the arc-transitivity.
\end{remark}

Now we focus on the orbit Cayley graphs over $Z_2^n$. Recall that $O_i$ are the orbits of $Z_2^{n}$ under the action of $W$ for $0\le i\le n$. Denote by $S_0=\cup_{i\ne 0,i\equiv 0(\mod 4)}O_i$, $S_1=\cup_{i\equiv 1(\mod 4)}O_i$, $S_2=\cup_{i\equiv 2(\mod 4)}O_i$ and $S_3=\cup_{i\equiv 3(\mod 4)}O_i$.
\begin{lem}\label{lem-z-9}
Let $v$ be an element of $Z_2^{4m}$. If $v\in S_0\cup S_1$ then
\[|C(v,S_0\cup S_1)|=2^{4m-2}+(-1)^m2^{2m-1}-2.\]
\end{lem}
\begin{proof} We divide two cases to discuss.

{\flushleft\bf Case 1.} $v\in S_0$.\\
In this case, assume that $v=4k$ where $1\le k\le m$. For $(x,y)\in C(v,S_0\cup S_1)$, we have $v=x+y$ and thus $x,y\in S_0$ or $x,y\in S_1$. Suppose that $x$ has $i$'s coordinates of $1$ coinciding with that of $v$, that is, $i=v^Tx$.

If $x,y\in S_0$ then $|x|=4t$ and $|y|=4s$ for some $0\le t,s\le m$ ($x,y\ne \mathbf{0}$). Note that $y=v+x$ due to $v=x+y$. By our assumption, $4s=|y|=|v+x|=(4k-i)+(4t-i)$. It follows that $i$ is even, say $i=2j$ where $0\le j\le r_{k,t}=\min\{2k,2t\}$. Therefore, $x$ has exactly ${4k\choose 2j}{4m-4k\choose 4t-2j}$ possible choices and $y$ is uniquely determined whenever $x$ is chosen. Note that $x,y\ne\mathbf{0}$ and so $(v,\mathbf{0})$ and $(\mathbf{0},v)$ should be excluded. Thus, there are exactly $\sum_{t=0}^m\sum_{j=0}^{r_{k,t}}{4k\choose 2j}{4m-4k\choose 4t-2j}-2=\sum_{t=0}^m\sum_{j=0}^{2k}{4k\choose 2j}{4m-4k\choose 4t-2j}-2$ pairs of such $(x,y)$ as ${4k\choose 2j}{4m-4k\choose 4t-2j}=0$ if $j>r_{k,t}$.

If $x,y\in S_1$ then $|x|=4t+1$ and $y=4s+1$ for some $0\le t,s\le m-1$. Note that $y=v+x$ due to $v=x+y$. By our assumption, $4s+1=|y|=|v+x|=(4k-i)+(4t+1-i)$. It follows that $i$ is also even, say $i=2j$ where $0\le j\le r_{k,t}=\min\{2k,2t\}$. Therefore, $x$ has exactly ${4k\choose 2j}{4m-4k\choose 4t+1-2j}$ possible choices and $y$ is uniquely determined whenever $x$ is chosen. Thus, there are exactly $\sum_{t=0}^m\sum_{j=0}^{r_{k,t}}{4k\choose 2j}{4m-4k\choose 4t+1-2j}=\sum_{t=0}^m\sum_{j=0}^{2k}{4k\choose 2j}{4m-4k\choose 4t+1-2j}$ pairs of such $(x,y)$ as ${4k\choose 2j}{4m-4k\choose 4t+1-2j}=0$ if $j>r_{k,t}$.

By arguments above, we have
\[\begin{array}{lll}|C(v,S_0\cup S_1)|&=&\displaystyle \sum_{t=0}^m\sum_{j=0}^{2k}{4k\choose 2j}{4m-4k\choose 4t-2j}+\sum_{t=0}^m\sum_{j=0}^{2k}{4k\choose 2j}{4m-4k\choose 4t+1-2j}-2\\
&=&\displaystyle\sum_{t=0}^m\sum_{j=0}^{2k}{4k\choose 2j}{4m-4k+1\choose 4t-2j+1}-2.\end{array}\]
By Theorem \ref{thm-z-2}, we have
\[|C(v,S_0\cup S_1)|=2^{4m-2}+(-1)^m2^{2m-1}-2.\]

{\flushleft\bf Case 2. } $v\in S_1$.\\
In this case, assume that $|v|=4k+1$ where $0\le k\le m-1$. For $(x,y)\in C(v,S_0\cup S_1)$, we have $v=x+y$ and thus $x\in S_0$, $y\in S_1$ or $x\in S_1$, $y\in S_0$. Suppose that $x$ has $i$'s coordinates of $1$ coinciding with that of $v$, that is, $i=v^Tx$.

If $x\in S_0$ and $y\in S_1$ then $|x|=4t$ and $y=4s+1$ for some $0\le t\le m$ and $0\le s\le m-1$. Note that $y=v+x$ due to $v=x+y$. By our assumption, $4s+1=|y|=|v+x|=(4k+1-i)+(4t-i)$. It follows that $i$ is even, say $i=2j$ where $0\le j\le r_{k,t}=\min\{2k,2t\}$. Therefore, $x$ has exactly ${4k+1\choose 2j}{4m-4k-1\choose 4t-2j}$ possible choices and $y$ is uniquely determined when $x$ is chosen. Note that $x\ne \mathbf{0}$ and so $(\mathbf{0},v)$ should be excluded. Thus, there are exactly $\sum_{t=0}^m\sum_{j=0}^{r_{k,t}}{4k+1\choose 2j}{4m-4k-1\choose 4t-2j}-1=\sum_{t=0}^m\sum_{j=0}^{2k}{4k+1\choose 2j}{4m-4k-1\choose 4t-2j}-1$ pairs of such $(x,y)$ as ${4k+1\choose 2j}{4m-4k-1\choose 4t-2j}=0$ if $j>r_{k,t}$. By the symmetry of $x$ and $y$, there are also $\sum_{t=0}^m\sum_{j=0}^{2k}{4k+1\choose 2j}{4m-4k-1\choose 4t-2j}-1$ pairs of $(x,y)$ such that $v=x+y$, $x\in S_1$ and $y\in S_0$.

By arguments above, we have
\[|C(v,S_0\cup S_1)|=2\sum_{t=0}^m\sum_{j=0}^{2k}{4k+1\choose 2j}{4m-4k-1\choose 4t-2j}-2.\]
By Theorem \ref{thm-z-3} (i), we have
\[|C(v,S_0\cup S_1)|=2^{4m-2}+(-1)^m2^{2m-1}-2.\]

It completes the proof.
\end{proof}
\begin{lem}\label{lem-z-10}
Let $v$ be an element of $Z_2^{4m}$. If $v\in \overline{S_0\cup S_1}=S_2\cup S_3$, then \[|C(v,S_0\cup S_1)|=2^{4m-2}+(-1)^m2^{2m-1}.\]
\end{lem}
\begin{proof}
We divide two cases to discuss.
{\flushleft\bf Case 1. } $v\in S_2$.\\
In this case, assume that $|v|=4k+2$ where $0\le k\le m-1$. For $(x,y)\in C(v,S_0\cup S_1)$, we have $v=x+y$ and thus $x,y\in S_0$ or $x,y\in S_1$. Suppose that $x$ has $i$'s coordinates of $1$ coinciding with that of $v$, that is, $i=v^Tx$.

If $x,y\in S_0$ then $|x|=4t$ and $|y|=4s$ for some $0\le t,s\le m$. Note that $y=v+x$ due to $v=x+y$. By our assumption, $4s=|y|=|v+x|=(4k+2-i)+(4t-i)$. It follows that $i$ is odd, say $i=2j+1$ where $0\le j\le r_{k,t}=\min\{2k,2t\}$. Therefore, $x$ has exactly ${4k+2\choose 2j+1}{4m-4k-2\choose 4t-2j-1}$ possible choices and $y$ is uniquely determined whenever $x$ is chosen. Thus, there are exactly $\sum_{t=0}^m\sum_{j=0}^{r_{k,t}}{4k+2\choose 2j+1}{4m-4k-2\choose 4t-2j-1}=\sum_{t=0}^m\sum_{j=0}^{2k}{4k+2\choose 2j+1}{4m-4k-2\choose 4t-2j-1}$ pairs of such $(x,y)$.

If $x,y\in S_1$ then $|x|=4t+1$ and $|y|=4s+1$ for some $0\le t,s\le m$. Note that $y=v+x$ due to $v=x+y$. By our assumption, $4s+1=|y|=|v+x|=(4k+2-i)+(4t+1-i)$. It follows that $i$ is also odd, say $i=2j+1$ where $0\le j\le r_{k,t}=\min\{2k,2t\}$. Therefore, $x$ has exactly ${4k+2\choose 2j+1}{4m-4k-2\choose 4t-2j}$ possible choices and $y$ is uniquely determined whenever $x$ is chosen. Thus, there are exactly $\sum_{t=0}^m\sum_{j=0}^{r_{k,t}}{4k+2\choose 2j+1}{4m-4k-2\choose 4t-2j}=\sum_{t=0}^m\sum_{j=0}^{2k}{4k+2\choose 2j+1}{4m-4k-2\choose 4t-2j}$ pairs of such $(x,y)$.

By arguments above, we have
\[\begin{array}{lll}|C(v,S_0\cup S_1)|&=&\displaystyle \sum_{t=0}^m\sum_{j=0}^{2k}{4k+2\choose 2j+1}{4m-4k-2\choose 4t-2j-1}+\sum_{t=0}^m\sum_{j=0}^{2k}{4k+2\choose 2j+1}{4m-4k-2\choose 4t-2j}\\
&=&\sum_{t=0}^m\sum_{j=0}^{2k}{4k+2\choose 2j+1}{4m-4k-1\choose 4t-2j}.\end{array}\]
By Theorem \ref{thm-z-3} (ii), we have
\[|C(v,S_0\cup S_1)|=2^{4m-2}+(-1)^m2^{2m-1}.\]

{\flushleft\bf Case 2. } $v\in S_3$.\\
In this case, assume that $|v|=4k+3$ where $0\le k\le m-1$. For $(x,y)\in C(v,S_0\cup S_1)$, we have $v=x+y$ and thus $x\in S_0$, $y\in S_1$ or $x\in S_1$, $y\in S_0$. Suppose that $x$ has $i$'s coordinates of $1$ coinciding with that of $v$, that is, $i=v^Tx$.

If $x\in S_0$ and $y\in S_1$ then $|x|=4t$ and $|y|=4s+1$ for some $0\le t\le m$ and $0\le s\le m-1$. Note that $y=v+x$ due to $v=x+y$. By our assumption, $4s+1=|y|=|v+x|=(4k+3-i)+(4t-i)$. It follows that $i$ is odd, say $i=2j+1$ where $0\le j\le r_{k,t}=\min\{2k+1,2t-1\}$. Therefore, $x$ has exactly ${4k+3\choose 2j+1}{4m-4k-3\choose 4t-2j-1}$ possible choices and $y$ is uniquely determined whenever $x$ is chosen. Thus, there are exactly $\sum_{t=0}^{m}\sum_{j=0}^{r_{k,t}}{4k+3\choose 2j+1}{4m-4k-3\choose 4t-2j-1}=\sum_{t=0}^{m}\sum_{j=0}^{2k}{4k+3\choose 2j+1}{4m-4k-3\choose 4t-2j-1}$ pairs of such $(x,y)$. By the symmetry of $x$ and $y$, there are also $\sum_{t=0}^{m}\sum_{j=0}^{2k}{4k+3\choose 2j+1}{4m-4k-3\choose 4t-2j-1}$ pairs of $(x,y)$ such that $v=x+y$ and $x\in S_1$ and $y\in S_0$.

By arguments above, we have
\[\displaystyle |C(v,S_0\cup S_1)|=2\sum_{t=0}^{m}\sum_{j=0}^{2k}{4k+3\choose 2j+1}{4m-4k-3\choose 4t-2j-1}.\]
By Theorem \ref{thm-z-3} (iii), we have
\[|C(v,S_0\cup S_1)|=2^{4m-2}+(-1)^m2^{2m-1}.\]

It completes the proof.
\end{proof}
Note that, by Lemma \ref{lem-z-7} (a), $|S_0\cup S_1|=\sum_{i=1}^m{4m\choose 4i}+\sum_{i=0}^m{4m\choose 4i+1}=\sum_{i=0}^m\left({4m\choose 4i}+{4m\choose 4i+1}\right)-1=\sum_{i=0}^m{4m+1\choose 4i+1}-1=2^{4m-1}+(-1)^m2^{2m-1}-1$. Combining Lemmas \ref{lem-z-9}, \ref{lem-z-10} and Theorem \ref{cor-z-3}, we have the following result.
\begin{thm}\label{thm-z-5}
The orbit Cayley graph $Cay(Z_2^{4m},S_0\cup S_1)$ is strongly regular with parameter
\[\left(2^{4m},2^{4m-1}+(-1)^m2^{2m-1}-1,2^{4m-2}+(-1)^m2^{2m-1}-2,2^{4m-2}+(-1)^m2^{2m-1}\right)\]
for any positive integer $m$.
\end{thm}
Note that the complement of a non-trivial $(n,r,\lambda,\mu)$-strongly regular graph is an $(n,n-1-r,n-2-2r+\mu,n-2r+\lambda)$-strongly regular \cite{Godsil}. Since $\overline{Cay(Z_2^{4m},S_0\cup S_1)}=Cay(Z_2^{4m},S_2\cup S_3)$, we get the following result.
\begin{cor}\label{cor-z-4}
The orbit Cayley graph $Cay(Z_2^{4m},S_2\cup S_3)$ is strongly regular with parameter
\[\left(2^{4m},2^{4m-1}-(-1)^m2^{2m-1},2^{4m-2}-(-1)^m2^{2m-1},2^{4m-2}-(-1)^m2^{2m-1}\right)\]
 for any positive integer $m$.
\end{cor}
By similar methods, we consider the orbit Cayley graphs $Cay(Z_2^{4m+2},S_0\cup S_1)$ and $Cay(Z_2^{4m+2},S_0\cup S_3)$.
\begin{lem}\label{lem-z-11}
Let $v$ be an element of $Z_2^{4m+2}$. If $v\in S_0\cup S_1$, then
\[|C(v,S_0\cup S_1)|=2^{4m}+(-1)^m2^{2m}-2.\]
\end{lem}
\begin{proof}
As similar as the proof of Lemma \ref{lem-z-9}, if $v\in S_0$, one can verify that there are $\sum_{t=0}^m\sum_{j=0}^{2k}{4k\choose 2j}{4m-4k+2\choose 4t-2j}-2$ pairs of $(x,y)\in C(v,S_0\cup S_1)$ such that $x,y\in S_0$ and there are $\sum_{t=0}^m\sum_{j=0}^{2k}{4k\choose 2j}{4m-4k+2\choose 4t+1-2j}$ pairs of $(x,y)\in C(v,S_0\cup S_1)$ such that $x,y\in S_1$. Therefore,  from Theorem (\ref{thm-z-3}) (iv), we have
\[\begin{array}{lll}
|C(v,S_0\cup S_1)|&=&\displaystyle\sum_{t=0}^m\sum_{j=0}^{2k}{4k\choose 2j}{4m-4k+2\choose 4t-2j}+\sum_{t=0}^m\sum_{j=0}^{2k}{4k\choose 2j}{4m-4k+2\choose 4t+1-2j}-2\\
&=&\displaystyle\sum_{t=0}^m\sum_{j=0}^{2k}{4k\choose 2j}{4m-4k+3\choose 4t-2j+1}-2\\
&=&2^{4m}+(-1)^m2^{2m}-2.\end{array}\]

If $v\in S_1$, there are $\sum_{t=0}^m\sum_{j=0}^{2k}{4k+1\choose 2j}{4m-4k+1\choose 4t-2j}-1$ pairs of $(x,y)\in C(v,S_0\cup S_1)$ such that $x\in S_0$ and $y\in S_1$. By symmetry of $x$ and $y$, there are also $\sum_{t=0}^m\sum_{j=0}^{2k}{4k+1\choose 2j}{4m-4k+1\choose 4t-2j}-1$ pairs of $(x,y)\in C(v,S_0\cup S_1)$ such that $y\in S_0$ and $x\in S_1$. Therefore, from Theorem \ref{thm-z-3} (v), we also have
\[|C(v,S_0\cup S_1)|=2\sum_{t=0}^m\sum_{j=0}^{2k}{4k+1\choose 2j}{4m-4k+1\choose 4t-2j}-2=2^{4m}+(-1)^m2^{2m}-2.\]

It completes the proof.
\end{proof}
\begin{lem}\label{lem-z-12}
Let $v$ be an element of $Z_2^{4m+2}$. If $v\in \overline{S_0\cup S_1}=S_2\cup S_3$, then
\[|C(v,S_0\cup S_1)|=2^{4m}+(-1)^m2^{2m}.\]
\end{lem}
\begin{proof}
As similar as the proof of Lemma \ref{lem-z-10}, if $v\in S_2$, there are $\sum_{t=0}^m\sum_{j=0}^{2k}{4k+2\choose 2j+1}{4m-4k\choose 4t-2j-1}+\sum_{t=0}^m\sum_{j=0}^{2k}{4k+2\choose 2j+1}{4m-4k\choose 4t-2j}$ pairs of $(x,y)\in C(v,S_0\cup S_1)$. By Theorem \ref{thm-z-3} (vi), we have
\[\begin{array}{lll}
|C(v,S_0\cup S_1)|&=&\displaystyle\sum_{t=0}^m\sum_{j=0}^{2k}{4k+2\choose 2j+1}{4m-4k\choose 4t-2j-1}+\sum_{t=0}^m\sum_{j=0}^{2k}{4k+2\choose 2j+1}{4m-4k\choose 4t-2j}\\
&=&\displaystyle\sum_{t=0}^m\sum_{j=0}^{2k}{4k+2\choose 2j+1}{4m-4k+1\choose 4t-2j}\\
&=&2^{4m}+(-1)^m2^{2m}.
\end{array}\]

If $v\in S_3$, there are $2\sum_{t=0}^m\sum_{j=0}^{2k+1}{4k+3\choose 2j+1}{4m-4k-1\choose 4t-2j-1}$ pairs of $(x,y)\in C(v,S_0\cup S_1)$. By Theorem \ref{thm-z-3} (vii), we also have
\[|C(v,S_0\cup S_1)|=2\sum_{t=0}^m\sum_{j=0}^{2k+1}{4k+3\choose 2j+1}{4m-4k-1\choose 4t-2j-1}=2^{4m}+(-1)^m2^{2m}.\]

It completes the proof.
\end{proof}

Note that, from  Lemma \ref{lem-z-7} (g), $|S_0\cup S_1|=\sum_{i=1}^m{4m+2\choose 4i}+\sum_{i=0}^m{4m+2\choose 4i+1}=\sum_{i=0}^m\left({4m+2\choose 4i}+{4m+2\choose 4i+1}\right)-1=\sum_{i=0}^m{4m+3\choose 4i+1}=2^{4m+1}+(-1)^m2^{2m}-1$. Combining  Lemmas \ref{lem-z-11}, \ref{lem-z-12} and Theorem \ref{cor-z-3}, we get the following result.
\begin{thm}\label{thm-z-6}
The orbit Cayley graph $Cay(Z_2^{4m+2},S_0\cup S_1)$ is strongly regular with parameter
\[\left(2^{4m+2},2^{4m+1}+(-1)^m2^{2m}-1,2^{4m}+(-1)^m2^{2m}-2,2^{4m}+(-1)^m2^{2m}\right)\]
for any positive integer $m$.
\end{thm}
Since $\overline{Cay(Z_2^{4m+2},S_0\cup S_1)}=Cay(Z_2^{4m+2},S_2\cup S_3)$, we have the following result.
\begin{cor}\label{cor-z-5}
The orbit Cayley graph $Cay(Z_2^{4m+2},S_2\cup S_3)$ is strongly regular with parameter
\[\left(2^{4m+2},2^{4m+1}-(-1)^m2^{2m},2^{4m}-(-1)^m2^{2m},2^{4m}-(-1)^m2^{2m}\right)\]
for any positive integer $m$.
\end{cor}
\begin{lem}\label{lem-z-13}
Let $v$ be an element of $Z_2^{4m+2}$. If $v\in S_0\cup S_3$, then
\[|C(v,S_0\cup S_3)|=2^{4m}-(-1)^m2^{2m}-2.\]
\end{lem}
\begin{proof}
If $v\in S_0$, one can verify that there are $\sum_{t=0}^m\sum_{j=0}^{2k}{4k\choose 2j}{4m-4k+2\choose 4t-2j}-2$ pairs of $(x,y)$ such that $v=x+y$ and $x,y\in S_0$, and there are $\sum_{t=0}^m\sum_{j=0}^{2k}{4k\choose 2j}{4m-4k+2\choose 4t+3-2j}$ pairs of $(x,y)$ such that $v=x+y$ and $x,y\in S_3$. Therefore, by Theorem \ref{thm-z-3} (viii), we have
\[\begin{array}{lll}
|C(v,S_0\cup S_3)|&=&\displaystyle\sum_{t=0}^m\sum_{j=0}^{2k}{4k\choose 2j}{4m-4k+2\choose 4t-2j}+\sum_{t=0}^m\sum_{j=0}^{2k}{4k\choose 2j}{4m-4k+2\choose 4t-2j+3}-2\\
&=&\displaystyle\sum_{t=0}^m\sum_{j=0}^{2k}{4k\choose 2j}{4m-4k+3\choose 4t-2j}-2\\
&=&2^{4m}-(-1)^m2^{2m}-2.
\end{array}\]

If $v\in S_3$, one can verify that there are $\sum_{t=0}^m\sum_{j=0}^{2k+1}{4k+3\choose 2j}{4m-4k-1\choose 4t-2j}-1$ pairs of $(x,y)$ such that $v=x+y$ and $x\in S_0$, $y\in S_3$, and there are also $\sum_{t=0}^m\sum_{j=0}^{2k+1}{4k+3\choose 2j}{4m-4k-1\choose 4t-2j}-1$ pairs of $(x,y)$ such that $v=x+y$ and $x\in S_3$, $y\in S_0$. Therefore, by Theorem \ref{thm-z-3} (ix), we have
\[|C(v,S_0\cup S_3)|=2\sum_{t=0}^m\sum_{j=0}^{2k+1}{4k+3\choose 2j}{4m-4k-1\choose 4t-2j}-2=2^{4m}-(-1)^m2^{2m}-2.\]

It completes the proof.
\end{proof}
\begin{lem}\label{lem-z-14}
Let $v$ be an element of $Z_2^{4m+2}$. If $v\in \overline{S_0\cup S_3}=S_1\cup S_2$, then
\[|C(v,S_0\cup S_3)|=2^{4m}-(-1)^m2^{2m}.\]
\end{lem}
\begin{proof}
If $v\in S_1$, there are $\sum_{t=0}^m\sum_{j=0}^{2k}{4k+1\choose 2j+1}{4m-4k+1\choose 4t-2j-1}$ pairs of $(x,y)$ such that $v=x+y$ and $x\in S_0$, $y\in S_3$, and there are also $\sum_{t=0}^m\sum_{j=0}^{2k}{4k+1\choose 2j+1}{4m-4k+1\choose 4t-2j-1}$ pairs of $(x,y)$ such that $v=x+y$ and $x\in S_3$, $y\in S_0$. Therefore, by Theorem \ref{thm-z-3} (x), we have
\[|C(v,S_0\cup S_3)|=2\sum_{t=0}^m\sum_{j=0}^{2k}{4k+1\choose 2j+1}{4m-4k+1\choose 4t-2j-1}=2^{4m}-(-1)^m2^{2m}.\]

If $v\in S_2$, there are $\sum_{t=0}^m\sum_{j=0}^{2k}{4k+2\choose 2j+1}{4m-4k\choose 4t-2j-1}$ pairs of $(x,y)$ such that $v=x+y$ and $x,y\in S_0$, and there are $\sum_{t=0}^m\sum_{j=0}^{2k}{4k+2\choose 2j+1}{4m-4k\choose 4t-2j+2}$ pairs of $(x,y)$ such that $v=x+y$ and $x,y\in S_3$. Therefore, by Theorem \ref{thm-z-3} (xi), we have
\[\begin{array}{lll}
|C(v,S_0\cup S_3)|&=&\displaystyle\sum_{t=0}^m\sum_{j=0}^{2k}{4k+2\choose 2j+1}{4m-4k\choose 4t-2j-1}+\sum_{t=0}^m\sum_{j=0}^{2k}{4k+2\choose 2j+1}{4m-4k\choose 4t-2j+2}\\
&=&\displaystyle\sum_{t=0}^m\sum_{j=0}^{2k}{4k+2\choose 2j+1}{4m-4k+1\choose 4t-2j+3}\\
&=&2^{4m}-(-1)^m2^{2m}.
\end{array}\]
It completes the proof.
\end{proof}

Note that, from Lemma \ref{lem-z-7} (f), $|S_0\cup S_3|=2^{4m+1}-(-1)^m2^{2m}-1$. Combining Lemmas \ref{lem-z-13}, \ref{lem-z-14} and Corollary \ref{cor-z-3}, we get the following result.
\begin{thm}\label{thm-z-7}
The orbit Cayley graph $Cay(Z_2^{4m+2},S_0\cup S_3)$ is strongly regular with parameter
\[\left(2^{4m+2},2^{4m+1}-(-1)^m2^{2m}-1,2^{4m}-(-1)^m2^{2m}-2,2^{4m}-(-1)^m2^{2m}\right)\]
for any positive integer $m$.
\end{thm}
Since $\overline{Cay(Z_2^{4m+2},S_0\cup S_3)}=Cay(Z_2^{4m+2},S_1\cup S_2)$, we have the following result.
\begin{cor}\label{cor-z-6}
The orbit Cayley graph $Cay(Z_2^{4m+2},S_1\cup S_2)$ is strongly regular with parameter
\[\left(2^{4m+2},2^{4m+1}+(-1)^m2^{2m},2^{4m}+(-1)^m2^{2m},2^{4m}+(-1)^m2^{2m}\right)\]
for any positive integer $m$.
\end{cor}

\section{Conclusion}
In this paper, we have constructed six families of infinite  non-trivial strongly regular graphs from the general linear group $GL(n,F_2)$, see Theorems \ref{thm-z-5}, \ref{thm-z-6}, \ref{thm-z-7} and Corollaries \ref{cor-z-4}, \ref{cor-z-5} and \ref{cor-z-6}. We collect them in Table \ref{tab-1}, and especially for $m=1$ and $2$, we present these graphs in Table \ref{tab-2}. Note that Brouwer \url{http://www.win.tue.nl/~aeb/graphs/srg/srgtab.html} has listed all strongly regular graphs on at most $1300$ vertices and thus the $12$ graphs in Table \ref{tab-2} are  contained in his collection. However, as we know, the six infinite families are new.

\begin{table}[htp]
\caption{The non-trivial strongly regular orbit Cayley graphs we have obtained.}
\begin{center}\setlength{\tabcolsep}{1mm}{{\footnotesize
\begin{tabular}{ccccc}
\hline
Graphs & $n$& $r$& $\lambda$ & $\mu$\\
\hline
$Cay(Z_2^{4m},S_0\cup S_1)$ & $2^{4m}$& $2^{4m-1}+(-1)^m2^{2m-1}-1$& $2^{4m-2}+(-1)^m2^{2m-1}-2$ & $2^{4m-2}+(-1)^m2^{2m-1}$\\
$Cay(Z_2^{4m},S_2\cup S_3)$ & $2^{4m}$& $2^{4m-1}-(-1)^m2^{2m-1}$&  $2^{4m-2}-(-1)^m2^{2m-1}$ & $2^{4m-2}-(-1)^m2^{2m-1}$\\
$Cay(Z_2^{4m+2},S_0\cup S_1)$ & $2^{4m+2}$& $2^{4m+1}+(-1)^m2^{2m}-1$& $2^{4m}+(-1)^m2^{2m}-2$ & $2^{4m}+(-1)^m2^{2m}$\\
$Cay(Z_2^{4m+2},S_2\cup S_3)$ & $2^{4m+2}$& $2^{4m+1}-(-1)^m2^{2m}$& $2^{4m}-(-1)^m2^{2m}$ & $2^{4m}-(-1)^m2^{2m}$\\
$Cay(Z_2^{4m+2},S_1\cup S_2)$ & $2^{4m+2}$& $2^{4m+1}+(-1)^m2^{2m}$& $2^{4m}+(-1)^m2^{2m}$ & $2^{4m}+(-1)^m2^{2m}$\\
$Cay(Z_2^{4m+2},S_0\cup S_3)$ & $2^{4m+2}$& $2^{4m+1}-(-1)^m2^{2m}-1$& $2^{4m}-(-1)^m2^{2m}-2$ & $2^{4m}-(-1)^m2^{2m}$\\
\hline
\end{tabular}}}
\end{center}
\label{tab-1}
\end{table}%

\begin{table}[htp]
\caption{The non-trivial strongly regular orbit Cayley graphs for $m=1$ and $2$.}
\begin{center}{\footnotesize
\begin{tabular}{ccccc||ccccc}
\hline
Graphs & $n$& $r$& $\lambda$ & $\mu$&Graphs & $n$& $r$& $\lambda$ & $\mu$ \\
\hline
$Cay(Z_2^{4},S_0\cup S_1)$ & $2^4$& $5$& $0$ & $2$&$Cay(Z_2^{8},S_0\cup S_1)$ & $2^8$& $135$& $70$ & $72$\\
$Cay(Z_2^{4},S_2\cup S_3)$ & $2^{4}$& $10$&  $6$ & $6$&$Cay(Z_2^{8},S_2\cup S_3)$ & $2^{8}$& $120$&  $56$ & $56$\\
$Cay(Z_2^{6},S_0\cup S_1)$ & $2^6$& $27$& $10$ & $12$&$Cay(Z_2^{10},S_0\cup S_1)$ & $2^{10}$& $527$& $270$ & $272$ \\
$Cay(Z_2^{6},S_2\cup S_3)$ & $2^{6}$& $36$& $20$ & $20$&$Cay(Z_2^{10},S_2\cup S_3)$ & $2^{10}$& $496$& $240$ & $240$\\
$Cay(Z_2^{6},S_1\cup S_2)$ & $2^{6}$& $28$& $12$ & $12$&$Cay(Z_2^{10},S_1\cup S_2)$ & $2^{10}$& $528$& $272$ & $272$\\
$Cay(Z_2^{6},S_0\cup S_3)$ & $2^{6}$& $35$& $18$ & $20$&$Cay(Z_2^{10},S_0\cup S_3)$ & $2^{10}$& $495$& $238$ & $240$\\
\hline
\end{tabular}}
\end{center}
\label{tab-2}
\end{table}

It is clear that if a Cayley graph is arc-transitive with diameter $2$ then it must be strongly regular. However, it is always not easy to verify whether a Cayley graph is arc-transitive. In Theorem \ref{cor-z-3}, we present a simpler criterion for a Cayley graph to be strongly regular, which is also applicative for non arc-transitive Cayley graphs. Note that the graph $Cay(Z_2^4,S_0\cup S_1)$ is just the famous Clebsch graph, which is arc-transitive, and $Cay(Z_2^4, S_0\cup S_1)$ is just the special case for $m=1$ in the family of $Cay(Z_2^{4m},S_0\cup S_1)$. It is natural to ask whether $Cay(Z_2^{4m},S_0\cup S_1)$ is arc-transitive for any positive integer $m$. More general, we pose the following problem.
\begin{prob}
Does all strongly regular graphs in the six families are arc-transitive ?
\end{prob}
The automorphism group of a graph always reflects some combinatoric properties of this graph. Therefore, many mathematicians would like to investigate the automorphism groups of graphs, especially the automorphism groups of Cayley graphs, see \cite{Feng-1} and reference therein. Note that the automorphism group of the Clebsch graph $Cay(Z_2^4,S_0\cup S_1)$ is isomorphic to the Coxeter group $D_5$. In fact, every isomorphism between two connected induced subgraphs of the Clebsch graph can be extended to an automorphism of it. The following problem is proposed.
\begin{prob}
Determine the automorphism groups of the strongly regular graphs in the six families.
\end{prob}

\end{document}